\documentclass[11pt]{articlefederico}

\usepackage{comment}

\usepackage[left=1.5in,top=1.0in,right=1.5in]{geometry}

\usepackage{amsthm}
\usepackage{amsmath}
\usepackage{amssymb}
\usepackage{tikz}
\usepackage{cite}

\usepackage{graphicx}

\newtheorem{theorem}{Theorem}[section]
\newtheorem{corollary}[theorem]{Corollary}
\newtheorem{lemma}[theorem]{Lemma}
\newtheorem{proposition}[theorem]{Proposition}

\newtheorem{question}[theorem]{Question}
\newtheorem{observation}[theorem]{Observation}
\newtheorem{notation}[theorem]{Notation}

\theoremstyle{remark}
\newtheorem{example}[theorem]{Example}
\newtheorem{remark}[theorem]{Remark}

\theoremstyle{definition}
\newtheorem{definition}[theorem]{Definition}

\newcommand{\R}{\mathcal{R}}
\renewcommand{\S}{\mathcal{S}}

\newcommand{\A}{{\cal A}}

\newcommand{\Pyr}{\mathrm{Pyr}}
\newcommand{\wt}{\mathrm{wt}}

\begin{document}

\title{\textsf{
Moving robots efficiently using \\
the combinatorics of CAT(0) cubical complexes.}}

\author{\textsf{Federico Ardila\footnote{\textsf{San Francisco State University, San Francisco, CA, USA and Universidad de Los Andes, Bogot\'a, Colombia, federico@sfsu.edu.}}
\qquad Tia Baker \qquad Rika Yatchak\footnote{\textsf{San Francisco State University, San Francisco, CA, USA, tiab@mail.sfsu.edu, rika@mail.sfsu.edu.}
\textsf{This research was partially supported by the National Science Foundation CAREER Award DMS-0956178 (Ardila), the National Science Foundation Grant DGE-0841164 (Baker and Yatchak), and  the SFSU-Colombia Combinatorics Initiative. This work is based on the Master's theses of the second and third authors, written under the supervision of the first author.}}}}

\date{}

\maketitle

\begin{abstract}
Given a reconfigurable system $X$, such as a robot moving on a grid or a set of particles traversing a graph without colliding, the possible positions of $X$ naturally form a cubical complex $\S(X)$. When $\S(X)$ is a CAT(0) space, we can explicitly construct the shortest path between any two points, for any of the four most natural metrics: distance, time, number of moves, and number of steps of simultaneous moves.

Using Ardila, Owen, and Sullivant's result that CAT(0) cubical complexes are in correspondence with posets with inconsistent pairs (PIPs), we can prove that a state complex $\S(X)$ is CAT(0) by identifying the corresponding PIP. We illustrate this very general strategy with one known and one new example: Abrams and Ghrist's positive robotic arm on a square grid, and the robotic arm in a strip. We then use the PIP as a combinatorial ``remote control" to move these robots efficiently from one position to another.

%
%
\end{abstract}

\noindent

\section{\textsf{Introduction}}

There are numerous contexts in mathematics, robotics, and many other fields   where a discrete system moves according to local, reversible moves. For example, one might consider a robotic arm moving around a grid, a number of particles moving around a graph, or a phylogenetic tree undergoing local mutations. Abrams, Ghrist, and Peterson \cite{AG, GP} introduced the formalism of \emph{reconfigurable systems} to model a very wide variety of such contexts.

Perhaps the most natural and important question that arises is the \emph{motion-planning} or \emph{shape-planning} question: how does one efficiently get a reconfigurable system $X$ from one position to another one? A first approach is to study the \emph{transition graph} $G(X)$ of the system, whose vertices are the states of the system and whose edges correspond to the allowable moves between them. Abrams, Ghrist, and Peterson observed that $G(X)$ is the 1-skeleton of the \emph{state complex} $\S(X)$: a cubical complex whose vertices are the states of $X$, whose edges correspond to allowable moves, and whose cubes correspond to collections of moves which can be performed simultaneously. In fact, $\S(X)$ can be regarded as the space of all possible positions of $X$, including the positions in between states. 

The geometry and topology of the state complex $\S(X)$ can help us solve the motion-planning problem for the system $X$. More concretely, $\S(X)$ is locally non-positively curved for \emph{any} configuration system. \cite{AG, GP} This implies that each homotopy class of paths from $p$ to $q$ contains a unique shortest path. Furthermore, the state complex of \emph{some} reconfigurable systems is globally non-positively curved, or \emph{CAT(0)}. This stronger property implies that for any two points $p$ and $q$ there is a unique shortest path between them. Ardila, Owen, and Sullivant \cite{AOS} gave an explicit algorithm to find this path. 
(In fact, there are at least four natural ways to measure the efficiency of a path between $p$ and $q$: by its Euclidean length inside $\S(X)$, by the amount of time that we take to traverse it, by the number of moves we perform, or by the number of steps of simultaneous moves that we perform. We show in this paper that, if $\S(X)$ is CAT(0), then we can find an optimal path under any of these four metrics.)

It is therefore extremely useful to find out when a state complex $\S(X)$ is CAT(0). The first ground-breaking result in this direction is due to Gromov \cite{Gr}, who gave a topological-combinatorial criterion for this geometric property. He proved that a cubical complex is CAT(0) if and only if it is simply connected (a topological condition) and the link of every vertex is a flag simplicial complex (a combinatorial condition). Roller \cite{Ro} and Sageev \cite{Sa}, and Ardila, Owen, and Sullivant \cite{AOS} then gave two completely combinatorial descriptions of CAT(0) cubical complexes. These two descriptions are slightly different but equivalent, and they have different advantages depending on the context. Ardila, Owen, and Sullivant gave a bijection between (pointed) CAT(0) cubical complexes and certain simple combinatorial objects, which they called \emph{posets with inconsistent pairs} (PIPs). These objects had been studied earlier in the computer science literature, where they were called \emph{coherent event structures}. \cite{WN, S} We restate this characterization of CAT(0) cubical complexes in Theorem \ref{th:poset}, and use it in a crucial way throughout the paper.

In this paper, we put into practice the paradigm introduced by Ardila, Owen, and Sullivant to prove that cubical complexes are CAT(0). In principle, this method is completely general, and we illustrate it with one known and one new example of robotic arms: Abrams and Ghrist's positive robotic arm in a quadrant, and the robotic arm in a strip of length 1. We use Theorem \ref{th:poset} to prove that the state complexes of these robotic arms are CAT(0). It follows that we can find the optimal way to move these robots from one position to another one, under any of the four metrics.

The key to navigating $\S(X)$ efficiently is to use its PIP as a  combinatorial ``remote control" that governs our movement inside $\S(X)$. While the state complex $\S(X)$ is often too large to be computed explicitly, the corresponding PIP is much smaller and more manageable.

\bigskip

The paper is organized as follows. In Section \ref{sec:prelim} we review the basic definitions of reconfigurable systems, cubical complexes, CAT(0) spaces, and the combinatorial characterization of CAT(0) cubical complexes in terms of \emph{PIP}s (posets with inconsistent pairs). In Section \ref{sec:robots} we introduce our two main objects of study: the positive robotic arm on a quadrant, and the robotic arm on a strip. In Sections \ref{sec:quadrant} and \ref{sec:strip}, respectively, we prove that these two robotic arms give rise to CAT(0) cubical complexes by identifying their corresponding PIPs. In contrast, we show in Section \ref{sec:notCAT(0)} an example of a state complex that is not CAT(0). In Section \ref{sec:optimization} we discuss the four natural ways to measure a path between two positions $p$ and $q$ of a system $X$. When the state complex $\S(X)$ is CAT(0), we show how to find a shortest path between $p$ and $q$ under any of these metrics. 
\section{\textsf{Preliminaries}} \label{sec:prelim}

\subsection{\textsf{Reconfigurable systems and cubical complexes}}

Abrams, Ghrist, and Peterson \cite{AG, GP} introduced the general framework of \emph{reconfigurable systems} to study discrete systems which vary according to local, reversible moves. This model applies to a wide variety of contexts, such as a robot moving on a grid without self-intersecting \cite{AG}, a set of particles traversing a graph without colliding \cite{GP}, or a variable element of a right angled Coxeter group $(W,S)$ which is subsequently getting multiplied by generators in $S$. \cite{Sc} We now recall the basic definitions pertaining to reconfigurable systems \cite{AG, GP} and illustrate them with an example.

Let $\mathcal{G}=\left(V,E\right)$
be a graph and $\mathcal{A}$ be a set of labels. A \emph{state} is a labeling of the vertices of $\mathcal{G}$ by elements of $\mathcal{A}$. Roughly speaking, a \emph{reconfigurable system} is given by a collection of states, together with a given set of moves called \emph{generators} that one can perform to get from one state to another. More precisely:

\begin{definition} A \emph{generator} $\varphi$ for a local reconfigurable system consists of:
\begin{itemize}
\item[a.] A subset $SUP(\varphi)\subset V$ called the \emph{support} of $\varphi$.
\item[b.] A subset $TR(\varphi)\subset SUP(\varphi)$ called the \emph{trace} of $\varphi$.
\item[c.] An unordered pair of \emph{local states} $u_{0}^{loc}$ and $u_{1}^{loc}$, which are labelings of $SUP(\varphi)$ by elements of $\mathcal{A}$ that agree outside of $TR(\varphi)$:
\[
u_{0}^{loc}|_{SUP(\varphi) - TR(\varphi)} = u_{1}^{loc}|_{SUP(\varphi) - TR(\varphi)}.
\]
\end{itemize}
\end{definition}

\begin{definition} A generator $\varphi$ is \emph{admissible} at a state $u$ if $u|_{SUP(\varphi)} \in \{u_{0}^{loc}, u_{1}^{loc}\}$. In that case, the generator $\varphi$ acts on the state $u$ to give the state 
\[
\varphi[u]:= \left\{
\begin{array}{ccll}u   & \textrm{on} & G-SUP(\varphi) & \\
u^{loc}_{1-i}   & \textrm{on} & SUP(\varphi), & \textrm{where } u|_{SUP(\varphi)} = u_{i}^{loc}. 
\end{array}\right.
\]
\end{definition}

Physically, states $u$ represent different positions or configurations of our object of interest. Each generator $\varphi$ is a rule that allows us to move between states: the move $\varphi$ changes $u$ locally at $SUP(\varphi)$, switching it from $u_{0}^{loc}$ to $u_{1}^{loc}$ or vice versa. Naturally, to be able to apply $\varphi$, the state $u$ must look like one of these two labelings locally.
 
\begin{definition} A \emph{reconfigurable system} on a graph consists of a collection $\Phi$ of generators, together with a collection of states which is closed under all possible admissible actions of the generators.   
\end{definition}	
	
\begin{figure}[h]
		\centering
		\includegraphics[scale=.4]{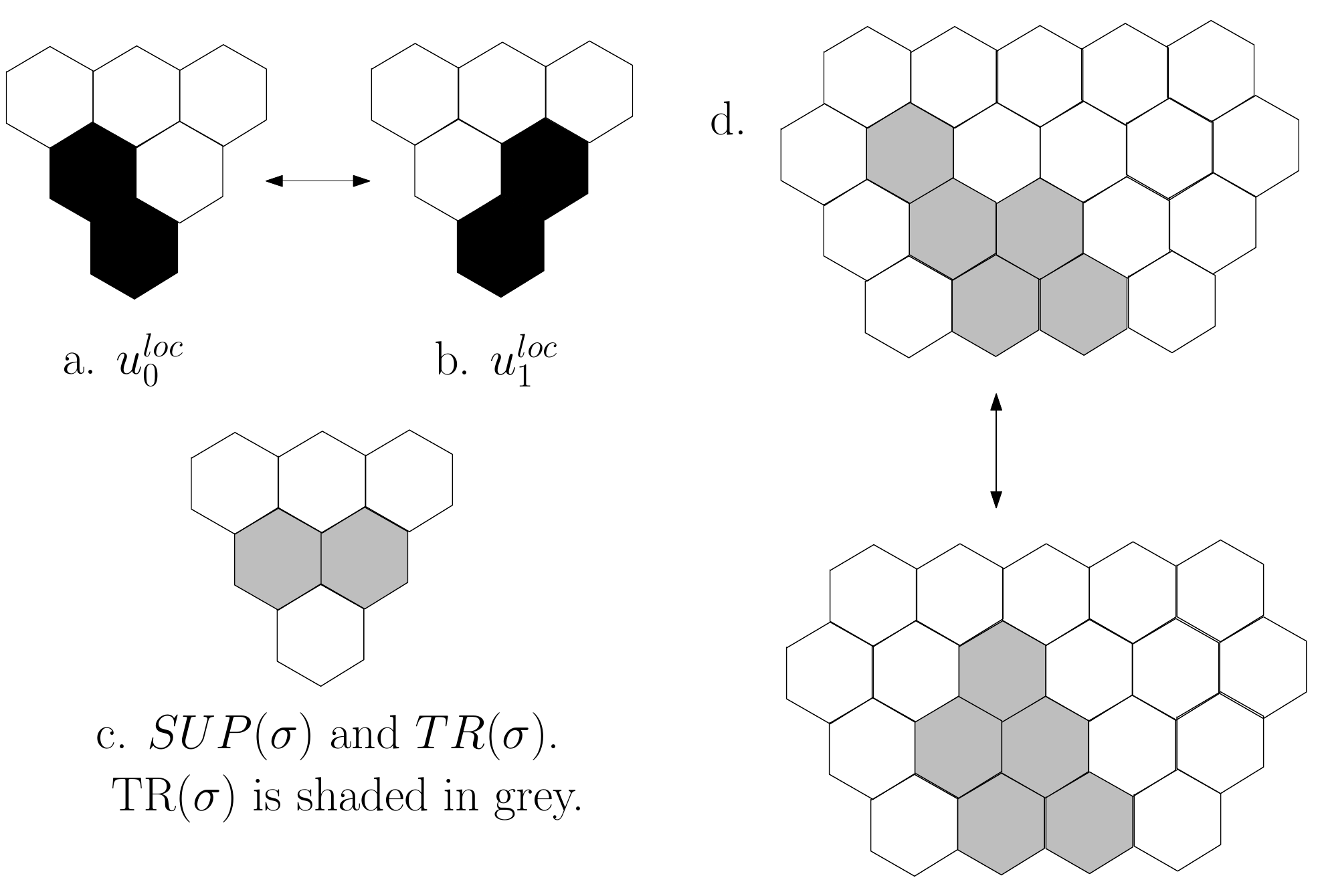}
		\caption{A generator for a metamorphic robot in the hexagonal lattice. \label{fig:hexrobot}}
		\end{figure}

\begin{example}[Metamorphic robots in a hexagonal lattice \cite{Ch, GP}]
Consider a robot made up of identical hexagonal unit cells in the hexagonal lattice, which has the ability to pivot cells on the boundary whenever they are unobstructed, as shown in Figure \ref{fig:hexrobot}a.-b.  
We can represent the hexagonal lattice by its dual graph $G$, where each unit hexagon is a vertex and we join neighboring vertices. The result is a triangular lattice. Then a state of the robot is just a labeling of the vertices of the triangular lattice $G$ with the set $\A=\{0,1\}$, where the $1$s represent the cells of the robot. Figure \ref{fig:hexrobot}a.-d. illustrate one generator $\varphi$: a. and b. show the two local states of $SUP(\varphi)$, c. shows $TR(\varphi)$ shaded (the cells where the pivoting takes place) inside $SUP(\varphi)$ (where we now include the cell necessary to pivot, and the three cells that could obstruct the pivoting), and d. shows the effect of applying $\varphi$ to one particular state.

\begin{definition}
A set $\{\varphi_1, \ldots, \varphi_n\}$ of generators \emph{commute} if for any $i \neq j$ we have that $TR(\varphi_i)\cap SUP(\varphi_j)=\emptyset$. In words,  generators commute when they are ``physically independent", so they can be applied simultaneously to any state where they are all admissible. 
\end{definition}

	\begin{figure}[h]
		\centering
		\includegraphics[scale=.3]{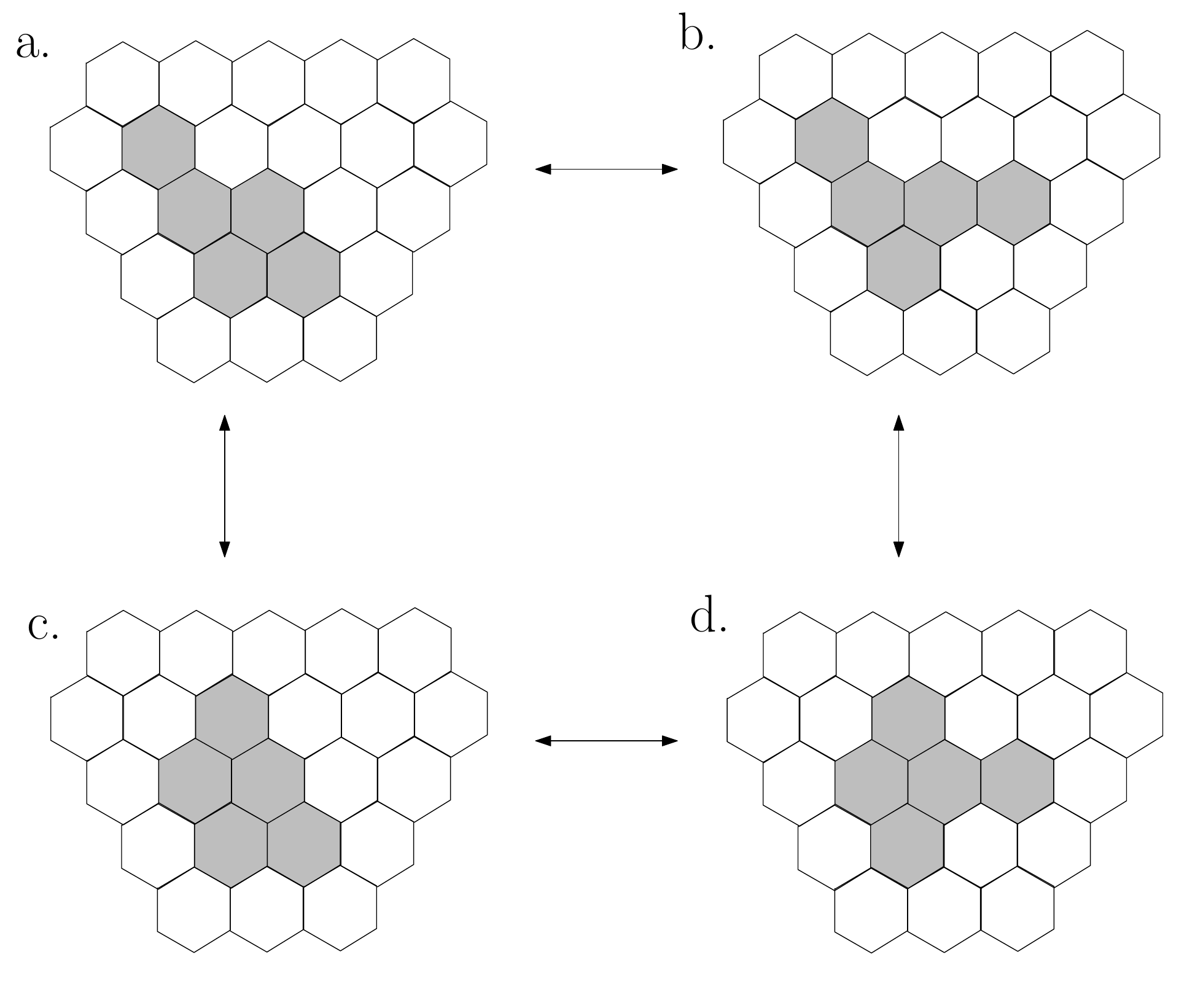}
		\caption{Four states and the two commuting generators connecting them.}\label{fig:hexrobotstates}
		\end{figure}
\end{example}

Figure \ref{fig:hexrobotstates} shows four states 
which are connected by two commutative moves. Note that we could reach state $d$ from $a$ by simultaneously performing these two moves. Cubical complexes are a useful tool to model this in general.
	
\begin{definition}
A \emph{cubical complex} $X$ is a polyhedral complex obtained by gluing cubes of various dimensions, in such a way that the intersection of any two cubes is a face of both. Such a space $X$ has a natural piecewise Euclidean metric space as follows: Each cell is given the Euclidean metric of a cube of length $1$, and the distance between two points $p$ and $q$ in $X$ is the infimum among the lengths of all piecewise linear paths from $p$ to $q$ in $X$.
\end{definition}

Any reconfigurable system gives rise to a cubical complex:
		
\begin{definition} \label{def:state complex}The \emph{state complex} $\mathcal{S}(\R)$ of a reconfigurable system $\mathcal{R}$ is a cubical complex whose vertices correspond to the states of $\mathcal{R}$. We draw an edge between two states if they differ by an application of a single generator. The $k$-cubes correspond to $k$-tuples of commutative moves: 
Given $k$ such moves which are applicable at a state $u$, we can obtain $2^k$ different states from $u$ by performing a subset of these $k$ moves; these are the vertices of a $k$-cube in $\mathcal{S}(\R)$.
\end{definition}

\begin{example}
Figure \ref{fig:hexcomplex} shows the state complex of the robot of 5 cells which starts horizontal in the lower right corner of a hexagonal tunnel of width 3, and is constrained to stay inside that tunnel. Notice that, due to the definition of the moves in Figure \ref{fig:hexrobot}, the robot is not able to pivot to the top row or out of the lower right corner.
\end{example}
		
\begin{figure}[h]
\centering
\includegraphics[scale=.9]{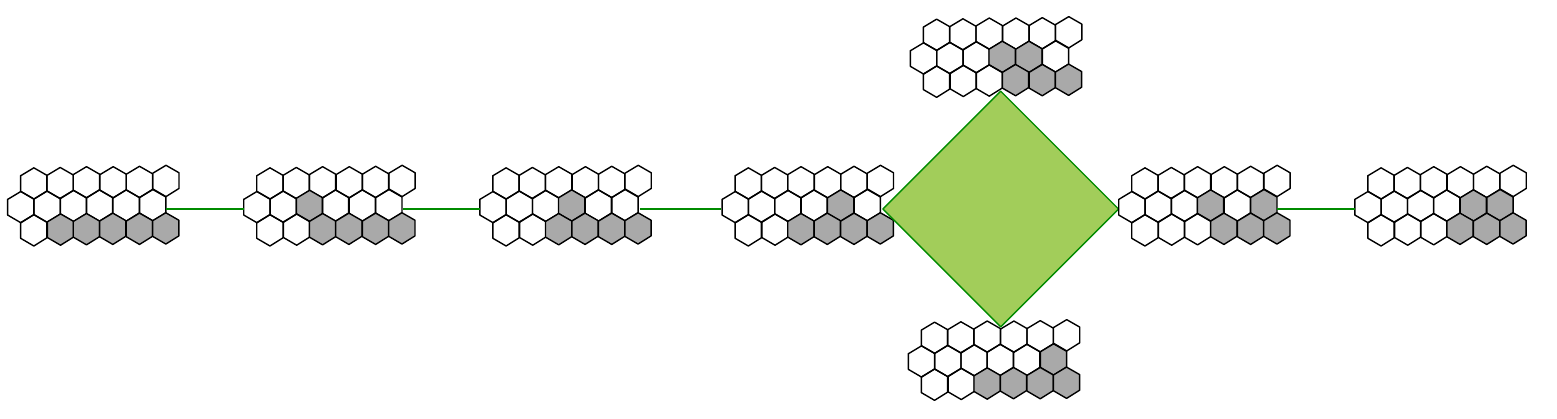}
\caption{The state complex of a hexagonal metamorphic robot in a tunnel.  \label{fig:hexcomplex}}
\end{figure}



The following observation is part of Definition \ref{def:state complex}; it is simple but important.  
\begin{observation}  \cite{GP}
Let $\R$ be a reconfigurable system. 
\begin{itemize} 
\item[a.]The 0-skeleton of the state complex $\S({\R})$ is the set of states $\R$.
\item[b.] The 1-skeleton of the state complex $\S({\R})$ is the transition graph of $\R$.
\end{itemize} 
\end{observation}

Given a reconfigurable system $\R$ and a state $u$, there is a natural partial order on the states of $\R$ as follows.

\begin{definition}\label{def:posetR}
Let $\R$ be a reconfigurable system and let $u$ be any ``home" state. Define a poset $\R_u$ on the set of states by declaring that $p \leq q$ if there is a shortest edge-path from the home state $u$ to $q$ going through $p$. More precisely, $p \leq q$ if  there is a sequence $u= p_0, p_1, \ldots, p_k=q$ of minimal length containing $p$, where each state $p_i$ can be obtained from $p_{i-1}$ by a single move.
\end{definition}

\subsection{\textsf{Combinatorial geometry of CAT(0) cubical complexes}}\label{sec:combo}

We now define CAT(0) spaces, the spaces of global non-positive curvature that we are interested in. For more information, see \cite{BBI, BH}. Let $X$ be a metric space where there is a unique geodesic (shortest) path between any two points. Consider a triangle $T$ in $X$ of side lengths $a,b,c$, and build a comparison triangle $T'$ with the same lengths in the Euclidean plane. Consider a chord of length $d$ in $T$ which connects two points on the boundary of $T$; there is a corresponding comparison chord in $T'$, say of length $d'$. If $d \leq d'$ for any chord in $T$, we say that $T$ is a \emph{thin triangle} in $X$.

\begin{figure}[h]
\centering
\includegraphics[scale=1]{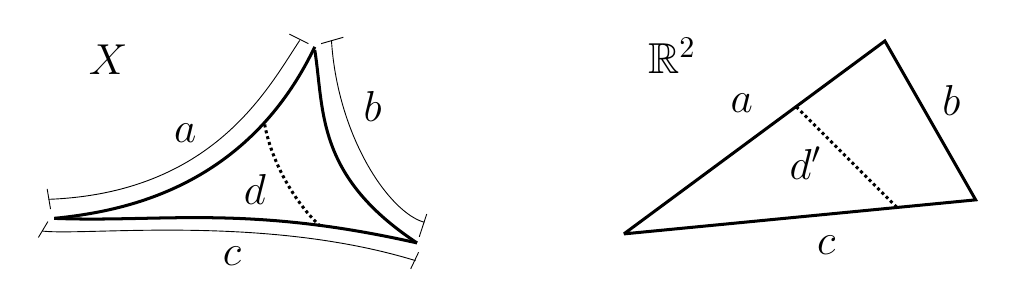}
\caption{A chord in a triangle in $X$, and the corresponding chord in the comparison triangle in the plane. The triangle in $X$ is \emph{thin} if $d \leq d'$ for all such chords. 
\label{fig:thintriangle}}
\end{figure}

\begin{definition} 
A \emph{CAT(0) space} is a metric space having a unique geodesic between any two points, such that every triangle is thin.
\end{definition}

A related concept is that of a \emph{locally CAT(0)} or \emph{non-positively curved} metric space $X$. This is a space where all sufficiently small triangles are thin.

Testing whether a general metric space is CAT(0) is quite subtle. However, Gromov \cite{Gr} proved that this is easier if the space is a cubical complex. In a cubical complex, the link of any vertex is a simplicial complex. We say that a simplicial complex $\Delta$ is \emph{flag} if it has no empty simplices; \emph{i.e.}, if any $d+1$ vertices which are pairwise connected by edges of $\Delta$ form a $d$-simplex in $\Delta$.

\begin{theorem}[Gromov]
A cubical complex is CAT(0) if and only if it is simply connected and the link of any vertex is a flag simplicial complex.
\end{theorem}

Ardila, Owen, and Sullivant \cite{AOS} gave a combinatorial description of CAT(0) cube complexes, which we now recall. This description may be seen as a global and purely combinatorial alternative to Gromov's theorem. 

If $X$ is a CAT(0) cubical complex and $v$ is any vertex of $X$, we call $(X, v)$ a \emph{rooted CAT(0) cubical complex}. The right side of Figure \ref{fig:bijection} shows an example; the cube and the three square flaps are all part of the complex. The vertex labels will become relevant later.

\begin{figure}[h]
\begin{center}
\includegraphics[width=1.5in]{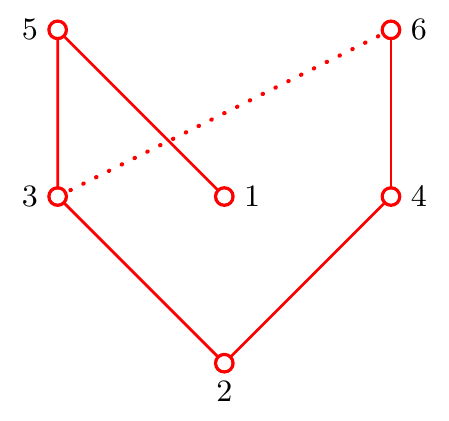} \qquad \qquad \qquad
\includegraphics[width=3in]{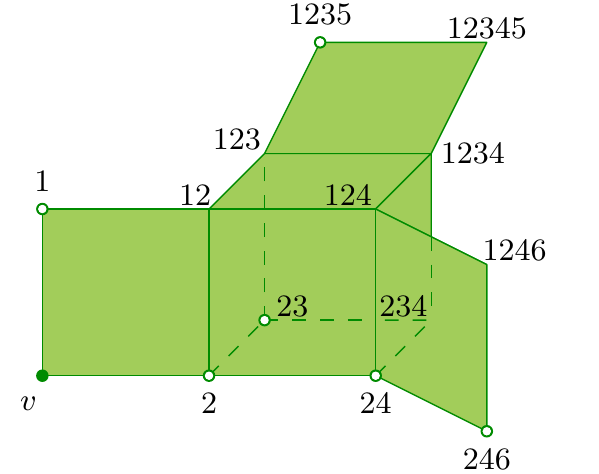}
\caption{A poset with inconsistent pairs and the corresponding rooted CAT(0) cubical complex.\label{fig:bijection} }
\end{center}
\end{figure}

Recall that a poset $P$ is \emph{locally finite} if every interval $[i, j] = \{k \in P \, : \, i \leq k \leq j\}$ is finite, and it has \emph{finite width} if every antichain (set of pairwise incomparable elements) is finite.

\begin{definition}
A \emph{poset with inconsistent pairs (PIP)} is a locally finite poset $P$ of finite width, together with a collection of \emph{inconsistent pairs} $\{p,q\}$, such that:
\begin{enumerate}
\item
If $p$ and $q$ are inconsistent, then there is no $r$ such that $r \geq p$ and $r \geq q$.
\item
If $p$ and $q$ are inconsistent and $p' \geq p$ and $q' \geq q$, then $p'$ and $q'$ are inconsistent.
\end{enumerate}
\end{definition}

\begin{remark}
Posets with inconsistent pairs are equivalent to \emph{coherent event structures}, defined earlier in the computer science literature; see for example \cite{WN, S}.
\end{remark}

The \emph{Hasse diagram} of a poset with inconsistent pairs (PIP) is obtained by drawing the poset, and connecting each minimal inconsistent pair with a dotted line. An inconsistent pair $\{p,q\}$ is \emph{minimal} if there is no other inconsistent pair $\{p',q'\}$ with $p' \leq p$ and $q' \leq q$. Naturally, the minimal inconsistent pairs determine all other inconsistent pairs. For example, the left side of Figure \ref{fig:bijection} shows the Hasse diagram of a PIP whose inconsistent pairs are $\{3,6\}$ and $\{5,6\}$. 

Recall that $I \subseteq P$ is an \emph{order ideal} if $a \leq b$ and $b \in I$ imply $a \in I$. 
A \emph{consistent order ideal} 
is one which contains no inconsistent pairs.


\begin{definition}\label{def:cubePIP}
If $P$ is a poset with inconsistent pairs, we construct the \emph{rooted cube complex of $P$}, which we denote $X(P)$. The vertices of $X(P)$ are identified with the consistent order ideals of $P$. 
There will be a cube $C(I,M)$ for each pair $(I, M)$ of a consistent order ideal $I$ and a subset $M \subseteq I_{max}$, where $I_{max}$ is the set of maximal elements of $I$. This cube has dimension $|M|$, and its vertices are obtained by removing from $I$ the $2^{|M|}$ possible subsets of $M$. The cubes are naturally glued along their faces according to their labels. The root is the vertex corresponding to the empty order ideal.
\end{definition}

Figure \ref{fig:bijection} shows a PIP $P$ and the corresponding complex  $X(P)$, which is rooted at $v$. For example, the compatible order ideal $I=\{1,2,3,4\}$ and the subset $M = \{1,4\} \subseteq I_{max}$ gives rise to the square with vertices labelled $1234, 123, 234, 23$.

\begin{theorem}[Ardila, Owen, Sullivant]\label{th:poset} \cite{AOS} 
The map $P \mapsto X(P)$ is a bijection between posets with inconsistent pairs and rooted CAT(0) cube complexes.
\end{theorem}

\subsection{\textsf{Reconfiguration systems and CAT(0) cubical complexes}}

The influential paper of Billera, Holmes, and Vogtmann \cite{BHV} was one of the first to highlight the relevance of the CAT(0) property in applications. 
A fundamental question in phylogenetics is to guess the most likely evolutionary tree of $n$ present day species, say by measuring how different their DNA sequences are. To approach this question, Billera, Holmes, and Vogtmann proposed a construction of the \emph{space of phylogenetic trees} $T_n$, a space whose points correspond to all the possible evolutionary trees for $n$ species. Many mathematical aspects of this phylogenetic question translate to understanding and efficiently navigating the space $T_n$. 

We now know that the space $T_n$ has close connections to important objects in algebraic geometry \cite{SS}, tropical geometry \cite{AK, Ar, SS}, topology \cite{Bo, Vo}, and combinatorics \cite{TZ}. Most relevantly to this paper, the space $T_n$ was shown in \cite{BHV} to be a CAT(0) cubical complex. This led to important consequences, such as the existence of unique geodesics and of ``average trees" in $T_n$. Furthermore, after numerous partial results by many authors, Owen and Provan \cite{OP} recently gave the first polynomial time algorithm to compute geodesics in $T_n$. 

The work of Billera, Holmes, and Vogtmann was generalized in the following two directions:

\begin{theorem}[Ardila-Owen-Sullivant] \cite{AOS}\label{th:algorithm}
There is an algorithm to compute the geodesic between any two points in a CAT(0) cubical complex.
\end{theorem}

\begin{theorem}[Abrams-Ghrist, Ghrist-Peterson] \cite{AG, GP} \label{th:NPC}
The state complex of a reconfigurable system is a \textbf{locally} CAT(0) cubical complex; that is, all small enough triangles are thin.
\end{theorem}

When the state complex of a reconfigurable system is \textbf{globally} CAT(0), we can use the algorithm in Theorem \ref{th:algorithm} to navigate it. That will allow us to get our system from one position to another one in the optimal way. This highlights the importance of the following question:

\begin{question}\label{q:CAT(0)?}
Is the state complex of a given reconfigurable system a CAT(0) space?
\end{question}

Theorem \ref{th:poset} offers a new technique to provide an affirmative answer to Question \ref{q:CAT(0)?}: Rooted CAT(0) cubical complexes are in bijection with PIPs; so to prove that a cubical complex is CAT(0), we ``simply"  have to choose a root for it, and find the corresponding PIP! The purpose of this paper is to put this paradigm to use for the first time, providing two concrete instances where it succeeds. We introduce the two relevant robots in Section \ref{sec:robots}, and provide combinatorial proofs that their state complexes are CAT(0) in Sections \ref{sec:quadrant} and \ref{sec:strip}. 

We remark that this technique is completely general. In principle, it works for \textbf{any} reconfigurable system whose state complex $X$ is CAT(0). In practice, it may not always be easy to identify the corresponding PIP. However, we hope to convince the reader that this can be done cleanly in many interesting special cases.

\section{\textsf{Two robotic arms}} \label{sec:robots}

In this section we introduce two robotic arms which have CAT(0) cubical complexes. The first robot was introducted by Abrams and Ghrist \cite{AG}, who proved that its state complex is CAT(0), using combinatorial and topological methods. We give a purely combinatorial proof in Section \ref{sec:quadrant}. To our knowledge, the second robot is new, and we use the same method in Section \ref{sec:strip} to prove that its state complex is also CAT(0).

\subsection{\textsf{The positive robotic arm in a quadrant}} \label{subsec:quadrant} 

The following model, which we call $QR_n$, was first introduced in \cite{AG}.  Consider a robotic arm consisting of $n$ links of unit length, attached sequentially. The robot lives inside an $n \times n$ grid, and its base is affixed to the lower left corner of the grid. Every one of its links must face north or east, starting from the base. The left panel of Figure \ref{fig:Qarm} shows a position of the arm. 

\begin{figure}[h]
\centering
\includegraphics[scale=1.5]{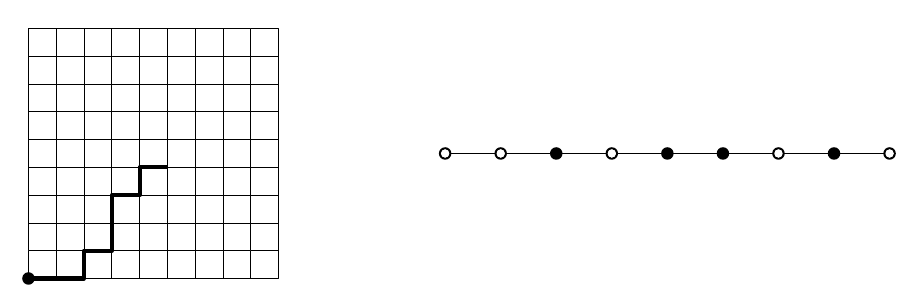} 
\caption{The robotic arm in position $3568$ for $n=9$, and the corresponding particles on a line (to be introduced later).\label{fig:Qarm}}
\end{figure}

The robot is free to move using the two local moves illustrated in Figure \ref{fig:QRmoves}:
\begin{itemize}
\item \emph{NE-Switching corners}: Two consecutive links facing north and east can be switched to face east and north, and vice versa.
\item \emph{NE-Flipping the end}: If the last link of the robot is facing east, it can be switched to face north, and vice versa.
\end{itemize} 
We call this the ``positive" robotic arm because its joints can only face north or east. 

\begin{figure}[h]
\centering
\includegraphics[scale=1.1]{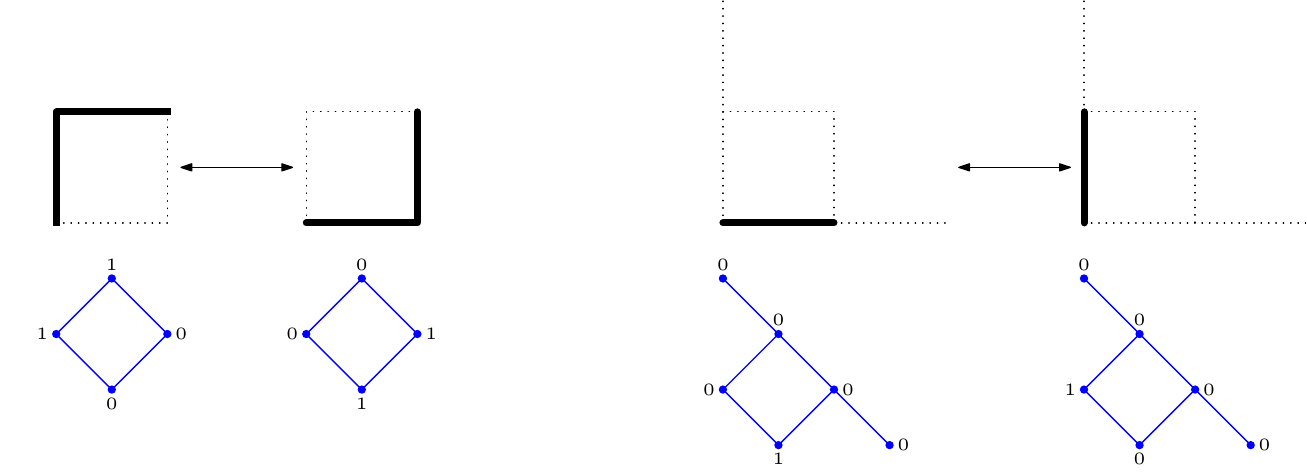} 
\caption{The two moves of $QR_n$.\label{fig:QRmoves}}
\end{figure}

It is clear that $QR_n$ has $2^n$ possible positions, corresponding to the paths of length $n$ which start at the southwest corner and always step east or north. We call these simply \emph{NE-paths}. 

\begin{notation}
We will label each state of the robot using the set of its vertical steps: if a position of the robot has $k$ links facing north at positions $a_1, \ldots, a_k$ (counting from the base), then we label it $\{a_1, \ldots, a_k\}$ or simply $a_1\ldots a_k$. 
\end{notation}

Notice that two states of different lengths can have the same label. We assume implicitly that the length of the robot is specified ahead of time.

\begin{proposition}
The system $QR_n$ is a reconfigurable system.
\end{proposition}

\begin{proof}
We can specify a state of the robotic arm by labelling the edges of the grid with $0$s and $1$s, where a $1$ indicates that an edge is occupied by the robot. The moves can be reinterpreted as shown in the bottom half of Figure \ref{fig:QRmoves}. The result follows from the definitions.
\end{proof}

\subsubsection{\textsf{The system $QR_n$ as hopping particles.}}\label{sec:Qparticles}
Consider a board consisting of $n$ slots on a line, and a system of indistinguishable particles hopping around the board. Any particle can hop to the slot immediately to its left or right whenever that slot is empty. Particles may enter the board by hopping onto the rightmost slot, and they may leave the board by hopping out of the rightmost slot. 

We say that two reconfigurable systems are \emph{equivalent} if they have isomorphic state complexes.

\begin{proposition} 
The system $QR_n$ is equivalent to the system of hopping particles on a board of length $n$.
\end{proposition}

\begin{proof}
There is an obvious bijection between the states of these two systems: If the robot is in position $\{a_1, \ldots, a_k\}$, we associate to it the state where there are $k$ particles at slots $a_1, \ldots, a_k$. Switching a NE-corner of the robot corresponds to moving a particle left or right, and flipping the end of the robot corresponds to a particle entering or leaving the system from the right. 
\end{proof}

This correspondence is illustrated in Figure \ref{fig:Qarm}.

\subsection{\textsf{The robotic arm in a strip}} \label{subsec:strip}

Now consider a robotic arm $SR_n$ which also consists of $n$ links of unit length, attached sequentially. Now the robot lives inside a $1 \times n$ grid, and its base is still affixed to the lower left corner of the grid, but the links do not necessarily have to face north and east. The left panel of Figure \ref{fig:SRparticles} shows a position of the arm. 

\begin{figure}[h]
\centering
\includegraphics[scale=1.5]{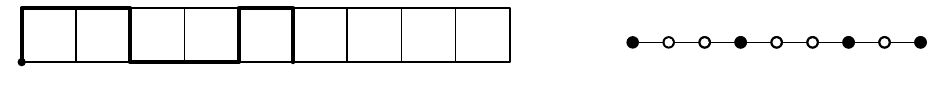} 
\caption{The robotic arm in position $1479$ for $n=9$, and the corresponding particles on a line. \label{fig:SRparticles}}
\end{figure}

The robot starts out fully horizontal, and is free to move using the local moves illustrated in Figure \ref{fig:SRmoves}:

\begin{itemize}
\item \emph{Switching corners}: Two consecutive links facing different directions can interchange their directions.
\item \emph{Flipping the end}: The end of the robot can rotate $90^{\circ}$ as long as it does not intersect the rest of the robot.
\end{itemize} 

\begin{figure}[h]
\centering
\includegraphics[scale=0.9]{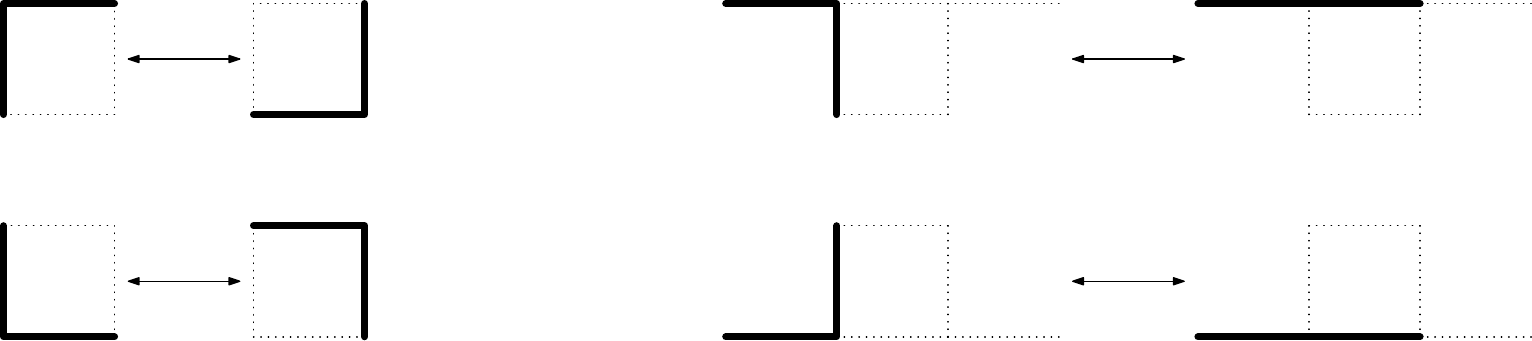} 
\caption{The four kinds of moves of $SR_n$.\label{fig:SRmoves}}
\end{figure}

The $SR_n$ robot is restricted to a smaller board, but it is not positive, so it has a wider range of moves. It can switch between north-east (NE) and east-north (EN) corners, as well as between south-east (SE) and east-south (ES) corners. The end of the robot can flip from facing east to facing either south or north. (No link ever faces west, due to the small height of the grid.)

\begin{proposition}
The system $SR_n$ is a reconfigurable system.
\end{proposition}

\begin{proof}
Again it is clear how to convert the moves of $SR_n$, shown in Figure \ref{fig:SRmoves}, to  the language of reconfigurable systems.
\end{proof}

\begin{lemma}
The number of possible positions of the robotic arm $SR_n$ is equal to the $(n+2)$-th Fibonacci number $F_{n+2} = \frac1{\sqrt 5}\left(\left(\frac{1+\sqrt 5}2 \right)^{n+2} - \left( \frac{1-\sqrt 5}2 \right)^{n+2} \right)$.
\end{lemma}

\begin{proof}
Let $a_n$ be the number of positions of $SR_n$. The positions whose first step faces east are in bijection with the positions of $SR_{n-1}$. On the other hand, if the first step of the robot faces north, then the second step must face east, and the rest of the robot -- after flipping it upside down -- is a position of $SR_{n-2}$. Therefore $a_n=a_{n-1}+a_{n-2}$. Since $a_1=2$ and $a_2=3$, the result follows.
\end{proof}

For this reason, we call a position of $SR_n$ a \emph{Fibonacci path}. We also say that a subset of $[n]$ is \emph{spread out} if it does not contain any two consecutive integers.

\begin{notation}
Again we label each state of the robotic arm using the set of its vertical steps: if a position of the robot has $k$ vertical links at positions $a_1, \ldots, a_k$ (counting from the base), then we label it $\{a_1, \ldots, a_k\}$ or simply $a_1\ldots a_k$. 
\end{notation}

For example, the label of the position of Figure \ref{fig:SRparticles} is $\{1,4,7,9\}$, or simply $1479$. It is clear how to recover the position of the arm from its label, so this gives a bijection between the states of $SR_n$ and the spread out subsets of $[n]$.

\subsubsection{\textsf{The system $SR_n$ as hopping repellent particles}}
\label{sec:Sparticles}

In analogy with what we did for $QR_n$, consider a board consisting of $n$ slots on a line, and a system of indistinguishable \emph{repellent particles} hopping around the board. The repellent particles must stay at distance at least $2$ from each other; so a particle can hop to the slot immediately to its left or right whenever that slot and its other neighbor  are empty. Particles may enter the board by hopping onto the rightmost slot, and they may leave the board by hopping out of the rightmost slot. 

\begin{proposition} 
The system $SR_n$ is equivalent to the system of hopping repellent particles on a board of length $n$.
\end{proposition}

\begin{proof}
Again, there is an obvious bijection between the states of these two systems: If the arm is in position $\{a_1, \ldots, a_k\}$, we associate to it the state where there are $k$ particles at slots $a_1, \ldots, a_k$. Switching a corner of the robot still corresponds to moving a particle left or right, and flipping the end of the robot still corresponds to a particle entering or leaving the system from the right. 
\end{proof}

This correspondence is illustrated in Figure \ref{fig:SRparticles}.

\subsubsection{\textsf{A positive robotic arm modeling $SR_n$}}

We will find it useful to construct a second robot, closer to $QR_n$ in spirit, which is equivalent to $SR_n$. Consider the \emph{pyramidal grid of size $n$}, which we denote $\Pyr_n$, consisting of strips of sizes $1 \times n, \, 1 \times (n-2), \, 1 \times (n-4), \ldots $ stacked on top of each other in decreasing order, so that each level is centered on top of the previous level. 

Now consider the positive robotic arm in the grid $\Pyr_n$ whose base is still affixed to the lower left corner of the grid, and whose links must point north or east. The robot moves using the following restricted versions of the two local moves of $QR_n$.
\begin{itemize}
\item \emph{NE-Switching corners}: 
Two consecutive links facing north and east \textbf{which are preceded and followed by east links} can be switched to face east and north, and vice versa
\item \emph{NE-Flipping the end}: If the last link of the robot is facing east and
 \textbf{it is preceded by an east link}, then it can be switched to face north, and vice versa.
\end{itemize} 

Now the possible positions of the robot correspond to the \emph{F-paths}: the paths in $\Pyr_n$ which start at the southwest corner and take steps north and east, without ever taking two consecutive steps north. The lower right corner of Figure \ref{fig:unfurl} shows an F-path of length $9$.

\begin{proposition}\label{prop:model2}
The system $SR_n$ is equivalent to the robotic arm in the pyramidal grid $\Pyr_n$.
\end{proposition}

\begin{proof}
Given a Fibonacci path of length $n$, we create a unique F-path by \emph{unfolding}, as follows. First we place the Fibonacci path in the bottom row of $\Pyr_n$, attached at the bottom left-corner. Now we follow the path from left to right. Whenever we encounter a vertical link, starting from the second one, we ``unfold" the arm into the next level of $\Pyr_n$ by taking the rest of the arm and flipping it vertically. We repeat this process until the path has no edges facing south; the result is a unique F-path. This construction is illustrated in Figure \ref{fig:unfurl}. It is clear how to recover the Fibonacci path from the F-path. 

\begin{figure}[h]
\centering
\includegraphics[scale=1.1]{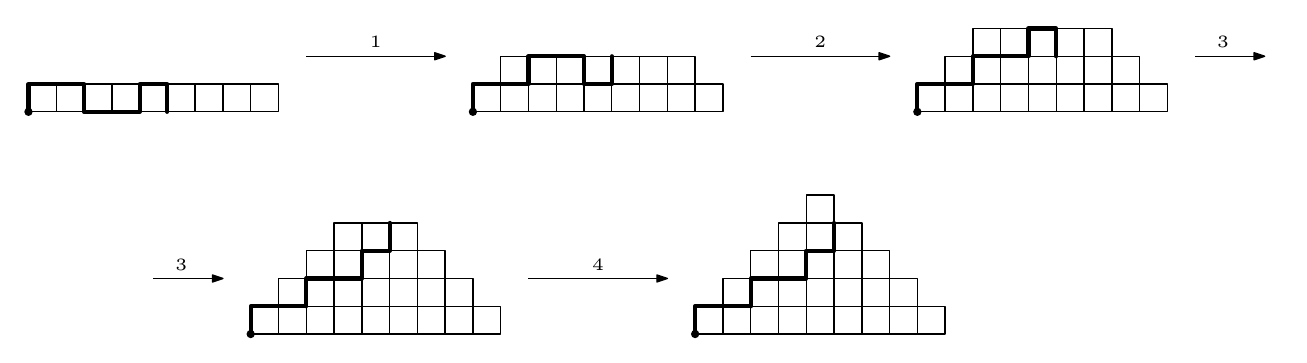} 
\caption{Unfolding the Fibonacci path $1479$ to obtain an F-path.\label{fig:unfurl}}
\end{figure}

One then easily checks that this bijection is compatible with the moves of the systems: switching the corners and flipping the end of $SR_n$ correspond to NE-switching the corners and flipping the end of the arm in $\Pyr_n$, respectively.
\end{proof}

\section{\textsf{The state complex of $QR_n$ is CAT(0)}} \label{sec:quadrant}

Having introduced all the background information, we now present a combinatorial proof that the state complex $\S(QR_n)$ of the robotic arm in a quadrant $QR_n$ is CAT(0). Abrams and Ghrist give a combinatorial-topological proof in \cite{AG}.

 In view of Theorem \ref{th:poset}, our strategy is as follows. We root the complex $\S(QR_n)$ at a natural vertex $v$, namely, the one corresponding to the fully horizontal robot. If $\S(QR_n)$ really is CAT(0), then Theorem \ref{th:poset} puts it in correspondence with a PIP (poset with inconsistent pairs) $QP_n$. By drawing the first few examples, it is not difficult to guess what the correct PIP should be in general. (For the two examples in this paper, it turns out that the PIP does not have any inconsistent pairs.) We then prove that, under the bijection of Theorem \ref{th:poset}, the PIP $QP_n$ is mapped to the (rooted) state complex of $QR_n$. Therefore this complex must be CAT(0).

\begin{definition}
Define the PIP $QP_n$ to be the set of lattice points inside the triangle $y \geq 0, y \leq x,$ and $x \leq n-1$, with componentwise order (so $(x,y) \leq (x',y')$ if $x \leq x'$ and $y \leq y'$) and no inconsistent pairs.
\end{definition}

The poset $QP_n$ has the triangular shape shown in Figure \ref{fig:QPn} for $n=6$.

\begin{figure}[h]
\centering
\includegraphics[scale=.8]{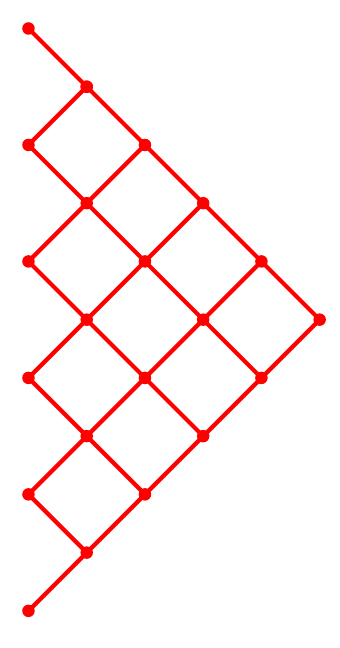} 
\caption{The poset $QP_6$. \label{fig:QPn}}
\end{figure}

\begin{proposition}\label{prop:bijQ}
There is a bijection between the states of the robot $QR_n$ and the order ideals of the poset $QP_n$.
\end{proposition}

\begin{proof}
Consider the $n(n+1)/2$ cells on the lower left half of the $n \times n$ board, below the main diagonal. Partially order them so that each cell is less than the cell to its left and the cell above it (if they exist). Clearly this order is isomorphic to $QP_n$; if we rotate the board $45^\circ$ clockwise, we will get the cells to line up with the Hasse diagram of this poset.

Any NE-path is contained in this lower left half of the board, and divides it into two parts. The part below the path clearly forms an order ideal of $QP_n$. Conversely, any order ideal comes from an NE-path in this way.
\end{proof}

\begin{figure}[h]
\centering
\includegraphics[scale=1]{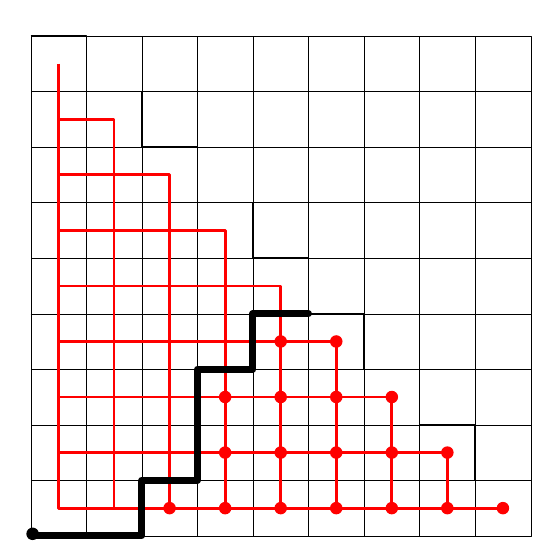} 
\caption{The bijection between states of $QR_n$ and order ideals of $QP_n$. \label{fig:Qbijection}}
\end{figure}

Recall that a \emph{lattice} is a poset where any two elements $x$ and $y$ have a least upper bound $x \vee y$ and a greatest lower bound $x \wedge y$. A lattice is \emph{distributive} if any three elements $x,y,z$ satisfy $x \vee (y \wedge z) = (x \vee y) \wedge (x \vee z)$ (which implies the dual relation.) An element is \emph{join-irreducible} if it is not the join of two smaller elements.

Birkhoff's theorem states that given a poset $P$, the poset of order ideals of $P$ ordered by containment is a distributive lattice. Conversely, every distributive lattice $L$ arises uniquely in this way from a poset $P$. In fact, $P$ can be recovered as the induced poset of join-irreducibles of $L$. For more information, see \cite{EC1}.

In light of Birkhoff's theorem, the following is immediate from Proposition \ref{prop:bijQ}.

\begin{corollary}
If we declare the ``home" state of $QR_n$ to be the fully horizontal state, then the poset of states of $QR_n$ is a distributive lattice. 
\end{corollary}

Let the \emph{word} \emph{of a subset} $A=\{a_1 < a_2 < \cdots < a_k\} \subseteq [n]$ be the length $n$ word $w(A) = (a_1, a_2,  \ldots a_k,  (n+1), (n+1),  \ldots,  (n+1))$.

\begin{proposition}\label{prop:Qorder}
The lattice of states of $QR_n$ is isomorphic to the poset on the subsets of $[n]$, where $A \leq B$ if $w(A) \geq w(B)$ coordinatewise.
\end{proposition}

\begin{proof}
This is clear from the hopping particles model for $QR_n$ of Section \ref{sec:Qparticles}. Each state corresponds to the subset of $[n]$ indicating the locations of the particles. We move up the poset by moving a particle to the left or having a particle enter from the right; and that precisely corresponds to decreasing a coordinate of $w(A)$ by one.
\end{proof}

\begin{remark}
Birkhoff's theorem tells us that the join-irreducible elements of the poset of states $QR_n$ should form a poset isomorphic to $QP_n$. A state of the hopping particles model is join-irreducible when only one particle can jump to the right. These are the states where the particles occupy positions $\{i, i+1, \ldots, j\}$ for some $i \leq j$. We can verify directly that these $n(n+1)/2$ states indeed from an isomorphic copy of the poset $QP_n$. Figure \ref{fig:Qjoinirreds} illustrates this in the example $n=4$. We have rotated the lattice $90^\circ$ clockwise to save space.
\end{remark}

\begin{figure}[h]
\centering
\includegraphics[scale=.9]{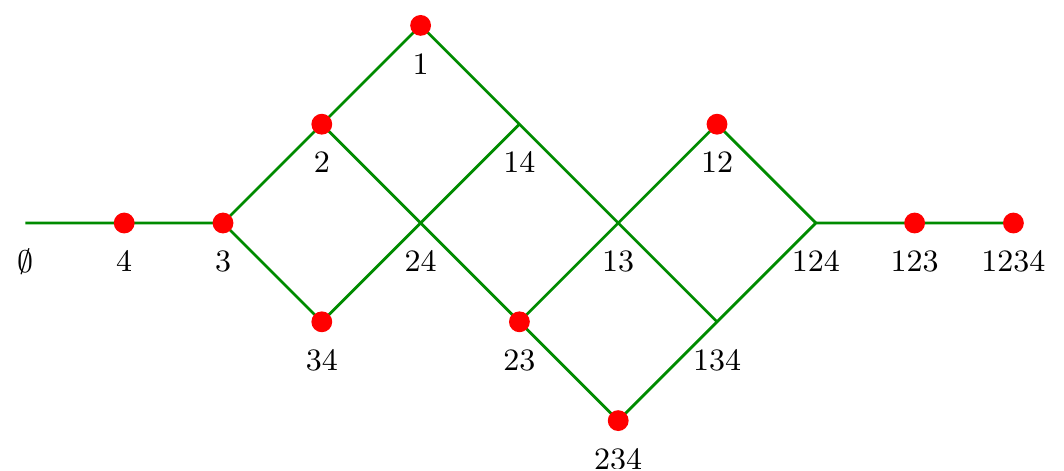} 
\caption{The join-irreducibles of the poset of $QR_n$ form a copy of $QP_n$.\label{fig:Qjoinirreds}}
\end{figure}

Having established these results about the 1-skeleton of the state complex, we now extend them to the higher-dimensional cubes.

\begin{definition}
A \emph{partial NE-path} is a path of consecutive links which may be north edges $N$, east edges $E$, unit squares $\square$, or partial unit squares $\llcorner$, such that 
\begin{itemize}
\item
Each unit square is attached to the rest of the path by its southwest and northeast corners. 
\item
There is at most one partial unit square $\llcorner$, which must be the last link, and must be attached to the rest of the path by its southwest corner.
\end{itemize}
The \emph{length} of a partial NE-path is $e+2f+g$, where $e$ is the number of edges, $f$ is the number of squares, and $g$ is the number of half-squares (which is $0$ or $1$). The partial NE-paths form a poset by containment, whose minimal elements are the NE-paths.
\end{definition}

To illustrate this definition, Figure \ref{fig:NEbijection2} shows a partial NE-path which contains the NE-path of Figure \ref{fig:Qbijection}.

Recall that $X(QP_n)$ is the rooted cube complex corresponding to the PIP $QP_n$ under the bijection of Theorem \ref{th:poset}. We use the notation of Definition \ref{def:cubePIP}.

\begin{lemma}\label{lemma:QP}
The partial NE-paths of length $n$ having $k$ squares or half-squares are in order-preserving bijection with the $k$-dimensional cubes of  $X(QP_n)$.
\end{lemma}

\begin{proof}
Consider a partial NE-path $p$. Its upper boundary is an NE-path corresponding to an order ideal $I \subseteq QP_n$, and its $k$ squares (possibly including a half-square at the end) correspond to $k$ maximal elements $m_1, \ldots, m_k$ of $I$. Let $p$ correspond to the cube $C(I, \{m_1, \ldots, m_k\})$ of $X(QP_n)$. This correspondence is illustrated in Figure \ref{fig:NEbijection2}. 
This is clearly an order-preserving bijection.
\end{proof}

Now recall that $\S(QR_n)$ is the state complex of the reconfigurable system $QR_n$.

\begin{lemma}\label{lemma:QR}
The partial NE-paths of length $n$ having $k$ squares or half-squares are in order-preserving bijection with the cubes of the state complex $\S(QR_n)$.
\end{lemma}

\begin{figure}[h]
\centering
\includegraphics[scale=1]{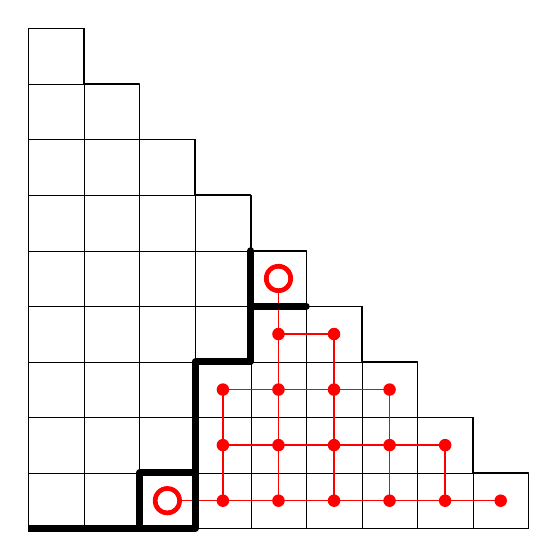} 
\caption{The bijection between partial NE-paths and pairs $(I,M)$ consisting of an order ideal $I$ and a subset $M$ of maximal elements of $I$ in $QP_n$. \label{fig:NEbijection2}}
\end{figure}

\begin{proof}
A $k$-cube $C$ of $\S(QR_n)$ is given by a state $u$ and a collection $\{\varphi_1, \ldots, \varphi_k\}$ of $k$ commutative moves that can be applied to $u$. The state $u$ is given by an NE-path, and each one of the $k$ moves $\varphi_1, \ldots, \varphi_k$ corresponds to a corner of the NE-path that could be switched. The two positions of this corner before and after the move $\varphi_i$ form a square (or half-square). Since the moves are commutative, two of these squares cannot share an edge.  Adding these $k$ squares to the NE-path $u$ gives rise to a partial NE-path corresponding to the $k$-cube $C$. 

Conversely, consider a partial NE-path with $k$ squares. There are $2^k$ NE-paths contained in it, obtained by ``resolving" each square into an NE or an EN corner, or ``resolving" the half-square into an N or an E step. The resulting $2^k$ NE-paths form a cube of $\S(QR_n)$. This bijection is clearly order-preserving. 
\end{proof}

\begin{theorem}\label{th:QR}
The state complex of the robotic arm $QR_n$ in an $n \times n$ grid is a CAT(0) cubical complex.
\end{theorem}

\begin{proof}
Lemmas \ref{lemma:QP} and \ref{lemma:QR} show that the state complex $\S(QR_n)$ is isomorphic to the cube complex $X(QP_n)$, which is CAT(0) by  Theorem \ref{th:poset}.
\end{proof}

\begin{proposition}
Let $q_{n,d}$ be the number of $d$-cubes in the state complex of the robot in a quadrant $QR_n$. Then
\[
\sum_{n, d \geq 0} q_{n,d}\, x^ny^d = \frac{1+xy}{1-2x-x^2y}
\]
\end{proposition}

\begin{proof}
The number $q_{n,d}$ counts the partial NE-paths of length $n$ having $d$ squares or half-squares. Say that such a path  $P$ has \emph{weight} $wt(P)=x^ny^d$. Each partial NE-path corresponds to a sequence of symbols $N,E,\square$, or {\large $\llcorner$}, where the symbol $\llcorner$ can only occur at the end of the sequence.

 Each $N$ contributes a weight of $x$, each $E$ contributes a weight of $x$, each $\square$ contributes a weight of $x^2y$, and each {\large $\llcorner$} contributes a weight of $xy$. Therefore 
\begin{eqnarray*}
\sum q_{n,d}\, x^ny^d &=& \sum_{\textrm{NE-paths } P} \wt(P) \\
&=& \sum_{\textrm{NE-paths } P \textrm{ with no } \llcorner} \wt(P) +
\sum_{\textrm{NE-paths } P \textrm{ with } \llcorner} \wt(P) \\
&=&  \frac{1}{1-\wt(N) - \wt(E) - \wt(\square)} + \frac{\wt({\large \llcorner})}{1-\wt(N) - \wt(E) - \wt(\square)} \\
&=& \frac{1+xy}{1-2x-x^2y}.
\end{eqnarray*}
Here we are using the observation that 
\[
\frac{1}{1-\wt(N) - \wt(E) - \wt(\square)} = \sum_{\stackrel{\textrm{ sequences } a_1 \ldots a_k \textrm{ with }}{a_i \in \{N,E,\square\}}} \wt(a_1)\cdots \wt(a_k)
\]
is the generating function for the weights of the NE-paths which do not include the symbol $\llcorner$. The result follows.
 \end{proof}

\begin{remark}
Notice that the PIP of this robot is much smaller than its state complex. While the state complex $QR_n$ has exponentially many faces, the size of the PIP $QP_n$ is quadratic in $n$.
\end{remark}

\section{\textsf{The state complex of $SR_n$ is CAT(0)}} \label{sec:strip}

Now we carry out the same approach for the robotic arm in a strip $SR_n$.

\begin{definition}
Define the PIP $SP_n$ to be the set of lattice points inside the triangle $y \geq 0, y \leq 2x,$ and $x \leq n-1$, with componentwise order (so $(x,y) \leq (x',y')$ if $x \leq x'$ and $y \leq y'$) and no inconsistent pairs.
\end{definition}

\begin{figure}[h]
\centering
\includegraphics[scale=.68]{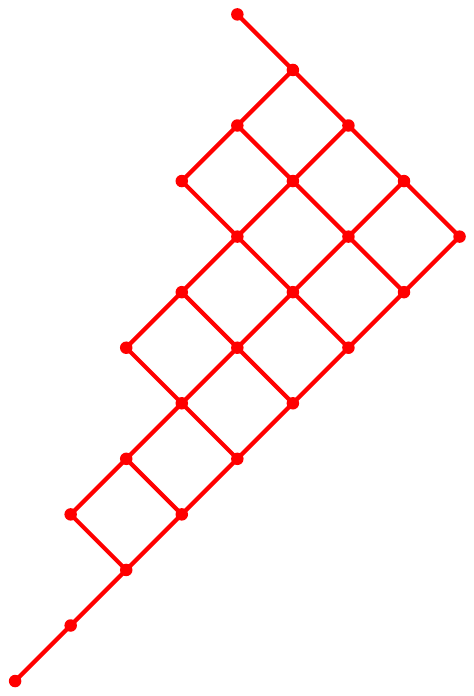} 
\caption{The poset $SP_n$. \label{fig:SPn}}
\end{figure}

The poset $SP_n$ has the triangular shape shown in Figure \ref{fig:SPn} for $n=9$. It is the union of a ``stack" of chains of lengths $n, n-2, n-4, \ldots$.

\begin{proposition}\label{prop:bijS}
There is a bijection between the states of the robot $SR_n$ and the order ideals of the poset $SP_n$.
\end{proposition}

\begin{proof}
We wish to imitate the proof of Proposition \ref{prop:bijQ}. It is not obvious how to do it for the robot in a strip, and this is the reason that we introduced the equivalent model of the robot in a pyramid in Proposition \ref{prop:model2}. 

Consider the $n + (n-2) + (n-4) + \cdots$ cells of the pyramidal board $\Pyr_n$. Partially order them so that each cell is less than the cell to its left and the cell northeast of it (if they exist), as shown in Figure \ref{fig:Sbijection}. This order is isomorphic to $SP_n$. Any F-path divides the pyramid into two parts. Since an F-path does not contain two consecutive steps north, the part of $\Pyr_n$ below it forms an order ideal of $QP_n$. Conversely, any order ideal comes from an F-path in this way.
\end{proof}

\begin{figure}[h]
\centering
\includegraphics[scale=1.3]{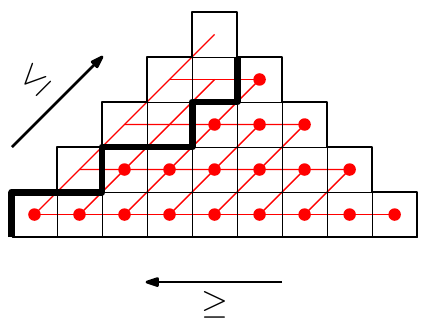} 
\caption{The bijection between states of $SR_n$ and order ideals of $SP_n$. \label{fig:Sbijection}}
\end{figure}

\begin{corollary}
If we declare the ``home" state of $SR_n$ to be the fully horizontal state, then the poset of states of $SR_n$ is a distributive lattice. 
\end{corollary}

Recall that a set of integers is \emph{spread out} if it contains no two consecutive integers.

\begin{proposition}
The lattice of states of $SR_n$ is isomorphic to the poset on the spread out subsets of $[n]$, where $A \leq B$ if $w(A) \geq w(B)$ coordinatewise.
\end{proposition}

\begin{proof} 
This is clear from the repellent hopping particles model for $SR_n$ of Section \ref{sec:Sparticles}, as in Proposition \ref{prop:Qorder}.
\end{proof}

\begin{remark}
The join-irreducible elements of the poset of $SR_n$ correspond to the states of the hopping particles model where only one particle can jump to the right. These are the states where the particles occupy positions $\{i, i+2, i+4,, \ldots, j\}$ for some $i \leq j$ of the same parity. These states form a copy of $QP_n$ inside the poset of $QR_n$, as illustrated in Figure \ref{fig:Sjoinirreds} for $n=6$. This lattice is also rotated $90^\circ$ clockwise.
\end{remark}

\begin{figure}[h]
\centering
\includegraphics[scale=.9]{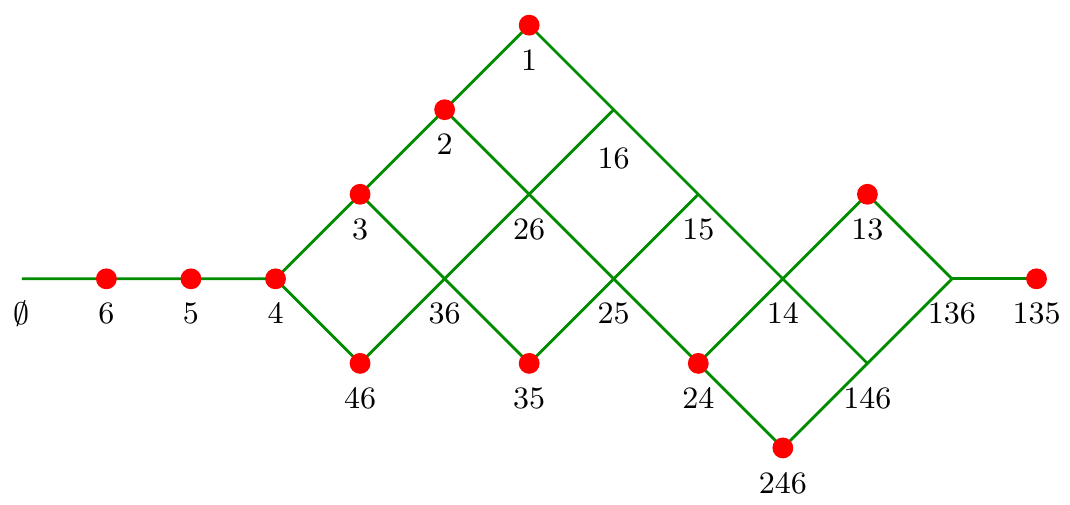} 
\caption{The join-irreducibles of the poset of $SR_n$ form a copy of $SP_n$.\label{fig:Sjoinirreds}}
\end{figure}

\begin{definition}
A \emph{partial F-path} is a partial NE-path such that 
the link following any vertical edge or square, if there is one, must be a horizontal edge.
\end{definition}

Recall that $X(SP_n)$ is the rooted cube complex corresponding to the PIP $SP_n$ under the bijection of Theorem \ref{th:poset}. We then have the following results. The proofs are essentially the same as those of Lemmas \ref{lemma:QP} and \ref{lemma:QR} and Theorem \ref{th:QR}.

\begin{lemma}\label{lemma:SP}
The partial F-paths of length $n$ having $d$ squares or half-squares are in bijection with the $d$-dimensional cubes of the state complex of $X(SP_n)$.
\end{lemma}
%

\begin{lemma}\label{lemma:SR}
The partial F-paths of length $n$ having $d$ squares or half-squares are in bijection with the $d$-dimensional cubes of the state complex of $SR_n$.
\end{lemma}


\begin{theorem}\label{th:SR}
The state complex of the robotic arm $SR_n$ of length $n$ in a strip of width 1 is a CAT(0) cubical complex.
\end{theorem}


\begin{proposition}
Let $s_{n,d}$ be the number of $d$-cubes in the state complex of the robot in a strip $SR_n$. Then
\[
\sum_{n, d \geq 0} s_{n,d}\, x^ny^d = \frac{1+x+xy+x^2y}{1-x-x^2-x^3y}.
\]
\end{proposition}

\begin{proof}
The number $s_{n,d}$ counts the sequences of $N,E$, $\square$, and {\large $\llcorner$} having weight $x^ny^d$ which contain none of the subwords $NN, N\square, \square N$, and $\square\square$ (so that an $N$ or an $\square$ can only be followed by an $E$), and which may only contain $\llcorner$ at the end of the sequence. 
Each such word may be regarded as a sequence of the ``clusters" $NE, \square E,$ and $E$, possibly followed by a single $N$, $\square$, or {\large $\llcorner$}. Therefore
\begin{eqnarray*}
\sum_{n, d \geq 0}
 s_{n,d}\, x^ny^d &=& \sum_{\textrm{F-paths } P} \wt(P) \\
&=&  \frac{1+\wt(N) + \wt(\square)+\wt({\large \llcorner})}{1-\wt(NE) - \wt(\square E) - \wt(E)}  \\
&=& \frac{1+x+xy+x^2y}{1-x^2 - (x^2y)x-x},
\end{eqnarray*}
as desired.
\end{proof}

\begin{remark}
The PIP of this robot is also quadratic in size, while the state complex is exponential in size.
\end{remark}

\section{\textsf{A robot whose state complex is not CAT(0)}}\label{sec:notCAT(0)}

It is worthwhile to exhibit an example of a similar robot whose state complex is not CAT(0). There are many such examples; \emph{e.g.}, see \cite{AG}. Here we consider an ``unpinned" version of the robotic arm in Section \ref{sec:strip}, which we can think of as a robotic snake. Now we are allowed to flip the end and switch any corners as long as we do not make the robot self-intersect. In general, the cubical complex of such an unpinned planar robotic snake is not CAT($0$).
			
Let $U^{\ell}_{m,n}$ denote the reconfigurable system corresponding to the unpinned robotic snake of length $\ell$ in an $m\times n$ grid, and $\S(U^{\ell}_{m,n})$ its corresponding state complex. The state complex  $\S(U^{1}_{m,n})$ is not CAT($0$) for any $n,m\in\mathbb{Z}$. In fact, one checks easily that $\S(U^{1}_{m,n})$ consists of $mn$ empty squares arranged diagonally on an $m \times n$ rectangular grid, glued corner to corner as shown in Figure \ref{fig:notCAT(0)}. This phenomenon can be observed for any given $\ell$ and sufficiently large $m$ and $n$.

\begin{figure}[h]
\centering
\includegraphics[scale=.65]{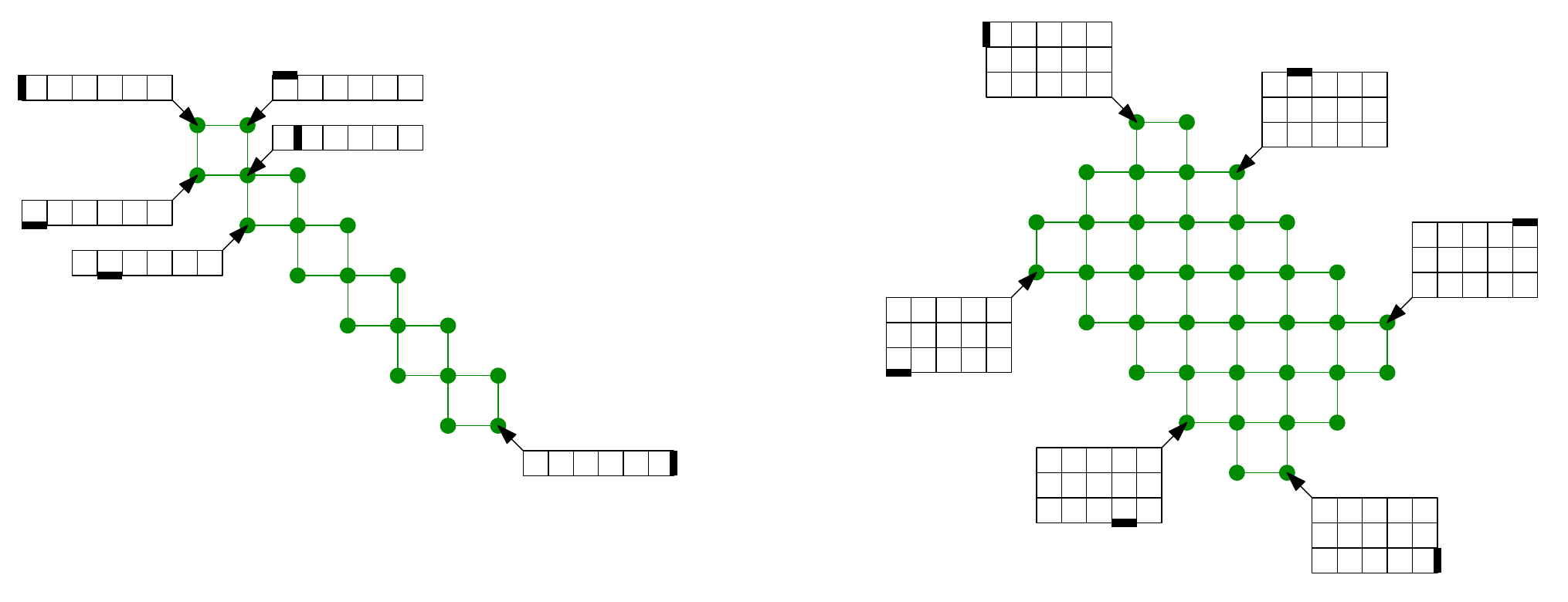}
\caption{The state complexes $SC^{1}_{1,6}$ and $SC^{1}_{3,5}$.\label{fig:notCAT(0)}}
\end{figure}
%
%

\section{\textsf{Finding the optimal path between two states}}\label{sec:optimization}

Consider a robot, or some other reconfigurable system $\R$, whose state complex $\S(\R)$ is CAT(0). As in the two examples above, there may be a natural choice of a ``home state" $u$, such that the PIP $P_u$ corresponding to the rooted complex $(\S(\R),u)$ has a particularly simple description. Now suppose that we want to take the robot from state $a$ to state $b$ in an optimal way. Equivalently, we wish to get from vertex $a$ to vertex $b$ of the state complex $\S(\R)$.

\subsection{\textsf{Rerooting the complex}}\label{subsec:reroot}

To find the optimal path from $a$ to $b$, the first step will be to reroot the complex at $a$, and find the PIP $P_a$ corresponding to the rooted CAT(0) cubical complex $(\S(\R),a)$. Fortunately, this is very easy to do.

\begin{notation}
If $p$ and $q$ are an inconsistent pair in a PIP, write $p \nleftrightarrow q$.
\end{notation}

\begin{proposition} \label{prop:reroot}
Let $u$ and $a$ be vertices of the CAT(0) cube complex $X$ and let $P_u$ and $P_a$ be the PIPs corresponding to the rooted complexes $(X,u)$ and $(X,a)$ respectively. Let $I$ be the consistent order ideal of $P_u$ corresponding to $a$, and let $J=P_u-I$. The PIP $P_a$ has an element $p'$ corresponding to each element $p \in P_u$, and it can be described in terms of $P_u$ as follows:

\begin{itemize}
\item If $j_1<j_2$ in $P_u$, then $j_1'<j_2'$ in $P_a$.
\item If $i_1<i_2$ in $P_u$ then $i_1'<i_2'$ in $P_a$.
\item If $i<j$ in $P_u$ then $i' \nleftrightarrow j'$ in $P_a$.
\item If $j_1 \nleftrightarrow j_2$ in $P_u$, then $j_1' \nleftrightarrow j_2'$ in $P_a$.
\item If $i \nleftrightarrow j$ in $P_u$ then $i' < j'$ in $P_a$.
\end{itemize}
Here the $i$s and the $j$s represent arbitrary elements of $I$ and $J$, respectively.\footnote{Notice that we never have $i>j$ or $i_1 \nleftrightarrow i_2$ in $P_u$.}
\end{proposition}

\begin{figure}[h]
\centering
\includegraphics[scale=1.3]{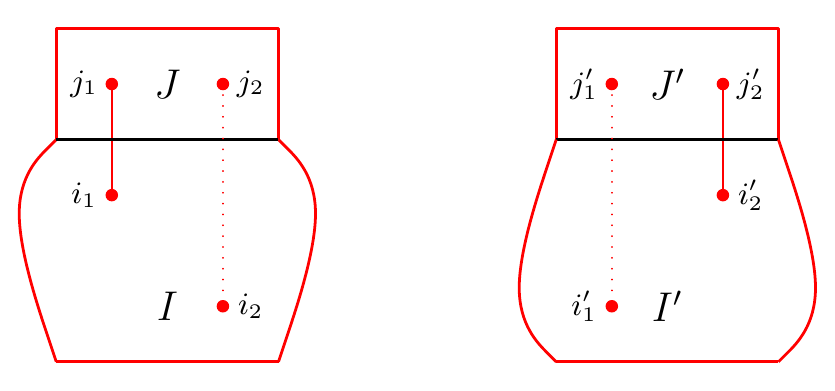}
\caption{The PIPs $P_u$ and $P_a$ before and after rerooting the CAT(0) cube complex.\label{fig:reroot}}
\end{figure}

\begin{corollary}\label{cor:reroot}
The Hasse diagram of $P_a$ is obtained from that of $P_u$ by 
\begin{itemize}
\item
turning $I$ upside down, and 
\item
converting all solid edges from $I$ to $J$ into dotted edges, and vice versa.
\end{itemize}
\end{corollary}
 
The effect of rerooting on the PIP is illustrated in Figure \ref{fig:reroot}. Notice that, even if $P_u$ has no inconsistent pairs, the PIP $P_a$ probably will have inconsistent pairs. (One can easily show that a CAT(0) cube complex has at most two roots whose PIPs have no inconsistent pairs, corresponding to the top and bottom elements of a distributive lattice.)

\begin{proof}[Proof of Proposition \ref{prop:reroot} and Corollary \ref{cor:reroot}]
Let us establish a bijection between the vertices of $X(P_u)$ and the vertices of $X(P_a)$. Given a consistent order ideal $B$ of $P_u$, write $B=B_I \cup B_J$ where $B_I = B \cap I$ and $B_J = B \cap J$. We claim that $B' = (I'-B_I') \cup B_J'$ is a consistent order ideal of $P_a$. (We write $S'=\{s' \, : \, s \in S\} \subseteq P_a$ for each  $S \subseteq P_u$.) 

First assume, for the sake of contradiction, that $B'$ is not an order ideal of $P_a$. Since $B_I$ and $B_J$ are ideals of $I$ and $J$ respectively, $I'-B_I'$ and $B_J'$ are ideals of $I'$ and $J'$ respectively. Therefore we must have $j' > i'$ for some $j' \in B_J'$ and $i' \notin (I'-B_I')$; that is, $i' \in B_I'$.  By the definition of $P_a$, we must then have $i \nleftrightarrow j$ in $P_u$ where $i,j \in B$, contradicting the consistency of $B$.  

Now assume that the order ideal $B'$ contains an inconsistent pair. Now, $I'$ contains no inconsistent pairs by the definition of $P_a$. If $B_J'$ contained an inconsistent pair, then $B_J$ would have to contain one as well. So the inconsistent pair must consist of $i' \in I' - B_I'$ and $j' \in B_J'$. But then $i < j$ for $i \notin B_I$ and $j \in B_J$, contradicting the fact that $B$ is an order ideal. This concludes the proof. 

The mapping $B \mapsto B'$ above establishes a bijection between the vertices of $X(P_u)$ and the vertices of $X(P_a)$. It is easy to see that adjacencies and cubes in $X(P_u)$ correspond to adjacencies and cubes in $X(P_u)$ as well. Therefore the PIP $P_a$ must be precisely the one we obtain by rerooting $X(P_u)=X(P_a)$ at $a$, as desired.

Having described the PIP $P_a$ explicitly in terms of $P_u$, it is straightforward to verify the description of its Hasse diagram.
\end{proof}

Now that we have rerooted the complex, our goal is to get from the root $a$ to the vertex $b$ in the optimal way. This will get our robot from position $a$ to position $b$ efficiently. There are at least four notions of ``optimality" for which we can do this.

\subsection{\textsf{Minimizing the Euclidean distance}}

Suppose we want to find the shortest path from $a$ to $b$ in the Euclidean metric of the cubical complex $\S(\R)$. This can be accomplished using  Ardila, Owen, and Sullivant's algorithm \cite{AOS} to compute the shortest path from $a$ to $b$. As explained there, a prerequisite for this is to write down the PIP $P_a$, which we have done in Proposition \ref{prop:reroot}.

This metric is very useful in some applications, particularly when navigating the space of phylogenetic trees \cite{BHV, OP}. However, this metric does not seem ideally suited to the robotic applications we have in mind here. For instance, suppose that the optimal path from $a$ to $b$ crosses a diagonal of a $d$-dimensional cube. To a robot, this means performing $d$ moves simultaneously. The cost of doing this in the Euclidean metric is $\sqrt{d}$, the length of the diagonal. In practice, it is conceivable that there may be a higher cost to performing $d$ moves simultaneously, but we do not know of a context where that cost would be $\sqrt{d}$. It is probably more natural to consider the following three variants.

\subsection{\textsf{Minimizing the number of moves}}

Suppose we are only allowed to perform one move at a time. Geometrically, we are looking for a shortest edge-path from $a$ to $b$. Let $B$ be the  consistent order ideal of $P_a$ corresponding to vertex $v$ in the rooted complex $(\S(\R),a)$. We can regard $B$ as a subposet of $P_a$.

\begin{proposition}
The shortest edge-paths from $a$ to $b$ are in one-to-one correspondence with the linear extensions of the poset $B$. Their length is $|B|$.
\end{proposition}

\begin{proof}
Combinatorially, a shortest path from $a$ to $b$ is simply obtained by building up the order ideal $B$ one element at a time, starting with the empty set, and adding a minimal missing element of $B$ at each step. 
\end{proof}

The previous argument also shows how to construct the minimal shortest paths.

\subsection{\textsf{Minimizing the sequence of simultaneous moves}}

Now suppose that we can move the robot in steps, where at each step we can  perform several moves at a time with no penalty. Geometrically, we are looking for a shortest cube path from $a$ to $b$, where at each step we cross a cube from the current vertex to the one across the diagonal. Again, let $B$ be the  consistent order ideal $B$ of $P_a$. 

Suppose that two opposite vertices $i$ and $j$ of a cube $C$ correspond to the order ideals $I$ and $J$ of $P_a$, respectively. Then  in the notation of Definition \ref{def:cubePIP}, the cube must be $C=C(I \cup J, I \triangle J)$, where $I \triangle J := (I-J) \cup (J-I)$ is the symmetric difference of $I$ and $J$. In particular, $I \triangle J$ is an antichain.

It follows that combinatorially, a cube path $a=v_0, v_1, \ldots, v_k=b$ corresponds to a sequence of order ideals $\mathbf{I}: \emptyset=I_0, I_1, \ldots, I_k=B$ such that the symmetric difference $I_j \triangle I_{j+1}$ is an antichain for all $j$. Let the \emph{depth} $d(B)$ of $B$ be the size of the longest chain(s) in $B$.

\begin{definition}\label{def:normal}
Let the \emph{normal cube path} from $a$ to $b$ be the cube path given by the sequence of order ideals $\mathbf{M} :  \emptyset = M_0 \subset M_1 \subset \cdots \subset M_{d(B)}=B$, where each ideal is obtained from the previous one by adding to it all the minimal elements that have not yet been added. In other words, $M_{k+1}:= M_{k} \cup  (B-M_{k})_{min}$ for $k \geq 0$.
\end{definition}
%

Alternatively, $M_k$ consists of the elements of $B$ whose shortest path to a minimal element has lengh $k-1$. From this it follows that $M_{d(B)}=B$ and $M_{d(B)-1} \neq B$.

The previous definition is due to Niblo and Reeves \cite{NR} in a different language; the correspondence with PIPs makes these paths more explicit. It also allows us to give a simple proof of the following result from Reeves's Ph.D. thesis \cite{Re}:

\begin{proposition}
The shortest cube paths from $a$ to $b$ have size $d(B)$. In particular, the normal cube path from $a$ to $b$ is minimal.
\end{proposition}


\begin{proof}
To prove that $\mathbf{M}$ is indeed minimal, we claim that any other cube path $\mathbf{I}$ satisfies $I_j \subseteq M_j \subsetneq B$ for all $j < d(B)$. We prove it by induction; the case $j=0$ is trivial. We assume that $I_j \subseteq M_j$, and prove that $I_{j+1} \subseteq M_{j+1}$. 

Consider an element $i \notin M_{j+1}$. Then we have $i \notin M_j$ (and therefore $i \notin I_j$) and $i \notin (B-M_j)_{min}$, which then implies that $i>i'$ for some $i' \in B-M_j$. This gives $i' \in B-I_j$, which implies that $i \notin (B-I_j)_{min}$. 
If we had $i \in I_{j+1}$, we would also have $i' \in I_{j+1}$, which would imply that the chain $i<i'$ is a subset of the antichain $I_j \triangle I_{j+1}$, a contradiction. Therefore $i \notin I_{j+1}$ as desired.
\end{proof}

If a robot has a CAT(0) state complex, the above results give us a way to move the robot from state $a$ to state $b$. First we compute the PIP $P_u$ corresponding to a ``natural" choice of a root $u$. This is the only non-trivial step, and its difficulty depends on the description of the reconfiguration system that we are given. Having found $P_u$, the rest is easy. We use Proposition \ref{prop:reroot} to reroot the complex to $a$ and find the corresponding PIP $P_a$. We then find the order ideal $B \subseteq P_a$, and compute the normal cube path of Definition \ref{def:normal}.

The shortest cube path from $a$ to $b$ is not necessarily unique. For instance, the ``reverse" normal cube path from $b$ to $a$ also has length $d(B)$, and is generally different. In fact, Abrams and Ghrist \cite[Algorithm 8.1]{AG}
gave a fast algorithm that starts with any given edge path from $a$ to $b$, and transforms it into a shortest cube path from $a$ to $b$. Their algorithm has the advantage that it does not require one to compute the whole state complex, which is often exponential in size. As we have explained, our approach offers an alternative way to overcome this difficulty. In the examples we have considered, and in most natural reconfiguration systems in the plane, a state complex of exponential size has a PIP of quadratic size.

\subsection{\textsf{Minimizing time}}

Perhaps the most realistic model is to allow ourselves to move the robot continuously in time, where we can perform several moves simultaneously, as long as these moves are physically independent. We can even perform only part of a move, and perform the rest of the move later. Each move still takes one unit of time, and there is no time penalty for multitasking.

Geometrically, we are endowing each cube with the $\ell_\infty$ metric: For $\mathbf{x}, \mathbf{y}$ in a unit $d$-cube, we let $||\mathbf{x} - \mathbf{y}|| := \max(x_1-y_1, \ldots, x_d-y_d)$. Now we are looking for a shortest path from $a$ to $b$ with respect to this $\ell_\infty$ metric. The following result, stated without proof in \cite{AG}, shows that the added flexibility of performing partial moves does not actually help us move our robots more quickly.

\begin{proposition}
The fastest paths from $a$ to $b$ take $d(B)$ units of time. In particular, the normal cube path from $a$ to $b$ is a fastest path.
\end{proposition}

\begin{proof}
Consider an optimal path from $a$ to $b$. When going from $a$ to $b$, we need to perform all the moves in $B$. 
Consider a longest chain $p_1 < \cdots < p_{d(B)}$ of $P_a$. For each $i$, we must spend a total of at least one unit of time performing the move $p_i$. Also, we can never perform (any part of) the moves $p_i$ and $p_j$ simultaneously, since they cannot be in the same antichain. Therefore we need at least $d(B)$ units of time to carry out these $d(B)$ moves.
It remains to remark that the normal cube path indeed takes $d(B)$ units of time.
\end{proof}

%
%

\section{\textsf{Acknowledgments}}
The authors would like to thank Megan Owen, Rick Scott, and Seth Sullivant for numerous enlightening discussions on this subject. They also thank the anonymous referees for their valuable suggestions.

\bibliographystyle{plain}

\end{document}